\documentclass{amsart}
\usepackage{amsfonts,amsmath,amsthm,amssymb}
\usepackage[linesnumbered,ruled,noend]{algorithm2e}
\usepackage{comment}
\usepackage{float}
\usepackage{latexsym}
\usepackage{graphicx}
\usepackage{multicol}
\usepackage{color}
\usepackage{pifont}
\usepackage{enumerate}
\usepackage{lmodern}
\usepackage{makecell}
\usepackage{adjustbox}
\usepackage[backref,pdftex=true,colorlinks,linkcolor=blue,citecolor=red,pdfstartview=FitH]{hyperref}
\usepackage{setspace}
\usepackage[all]{xy}
\usepackage{pb-diagram,pb-xy}
\usepackage{pdflscape}
\usepackage{chngpage}
\usepackage{breakurl}
\usepackage[dvipsnames]{xcolor}

\makeatletter
\def\BState{\State\hskip-\ALG@thistlm}
\makeatother
\theoremstyle{plain}
\newtheorem{theorem}{Theorem}[section]
\newtheorem{lemma}[theorem]{Lemma}
\newtheorem{proposition}[theorem]{Proposition}
\newtheorem{corollary}[theorem]{Corollary}

\theoremstyle{definition} 
\newtheorem{definition}[theorem]{Definition}

\newtheorem*{example*}{Example}
\newtheorem{remark}[theorem]{Remark}

\newcommand{\Q}{\mathbb{Q}}
\newcommand{\Z}{\mathbb{Z}}
\newcommand{\NN}{\mathbb{N}}

\newcommand{\p}{\mathfrak{p}}

\newcommand{\OO}{\mathcal{O}}

\begin{document}

\title[]{On the solutions to $Ax^p+By^p+Cz^p=0$ over quadratic fields}

\author{Alejandro Arg\'{a}ez-Garc\'{i}a}
\address{Facultad de Matemáticas, Universidad Aut\'{o}noma de Yucat\'{a}n. Perif\'{e}rico Norte Kil\'{o}metro 33.5, Tablaje Catastral 13615 Chuburna de Hidalgo Inn, M\'{e}rida, Yucat\'{a}n, M\'{e}xico. C.P. 97200 }
\email{alejandro.argaez@correo.uady.mx}

\author{Luis Elí Pech-Moreno}
\address{Facultad de Matemáticas, Universidad Aut\'{o}noma de Yucat\'{a}n. Perif\'{e}rico Norte Kil\'{o}metro 33.5, Tablaje Catastral 13615 Chuburna de Hidalgo Inn, M\'{e}rida, Yucat\'{a}n, M\'{e}xico. C.P. 97200  }
\email{luis.eli.pech@gmail.com}
\date{\today}


\begin{abstract}
We provide the necessary conditions for the existence of solutions $(x,y,z)$ to $Ax^p+By^p+Cz^p=0$ over any quadratic number field $K$ with $A,B,C$ pth powerfree integer numbers. We determine when $x$, $y$ and $z$ are rational numbers for pairwise coprime integers $A$, $B$ and $C$. Moreover, we prove that $x$, $y$ and $z$ are in $K\setminus\Q$ when $BC=\pm 1$ and $A\neq \pm 2$. Finally, we prove that no solutions $(x,y,z)$ to $Ax^p+By^p+Cz^p=0$ exist in $K\setminus\Q$ when $BC\neq \pm 1$.
\end{abstract}
\maketitle

%------------------------------------------
\section{Introduction} 
%------------------------------------------

The study of the Diophantine equation $Ax^p+By^p+Cz^p=0$ and its solutions has been of interest for a long time, particularly over different fields. A classical approach in which we could analyse its solutions is to, assuming we have a solution to it, construct either an elliptic curve or a hyperelliptic curve, depending on p, and classify the points on that curve. For example, when $p=3$, one can construct an elliptic curve over $\Q$ to determine the solutions to $x^3+y^3-kz^3=0$ in any quadratic field $\Q(\sqrt{d})$, \cite{ADP}.

In this article, we prove, using classical techniques, that the Diophantine $Ax^p+By^p+Cz^p$ over $\Q(\sqrt{d})$ does not have solutions in $K\backslash \Q$ when $BC\neq \pm1$. Our approach is to assume that the Diophantine equation has a nontrivial solution $(x,y,z)$, then construct the hyperelliptic curve $Y^2=X^p+\dfrac{A^2(BC)^{p-1}}{4}$ over that quadratic field and fully describe all possible points $(X,Y)$ in order to describe the initial solution $(x,y,z)$.

The following theorem and corollary summarise our results.

\begin{theorem}
Let $(x,y,z)$ be a solution to $Ax^p+By^p+Cz^p=0$ where $x$, $y$, $z$ are in $\Q(\sqrt{d})$ with $xyz\neq 0$, $A$, $B$, $C$ are $p$th powerfree coprime integers, $d$ is a squarefree integer, and $p>3$ is a prime number. Let $Y^2=X^p+\dfrac{A^2(BC)^{p-1}}{4}$ be the associated hyperelliptic curve. Then,
\begin{enumerate}[$(a)$]
\item $(x,y,z)$ is a rational solution when $A$, $B$ and $C$ are pairwise coprime, and $Y$ is a rational number.
\item $(x,y,z)=(x,u\overline{z},z)$ is a $K$-rational solution when $BC=\pm 1$ and $Y=n\sqrt{d}$, where $u$ is a unit and $n$ is a rational number.
\end{enumerate}
\end{theorem}

\begin{corollary}
There are no solutions $(x,y,z)$ to $Ax^p+By^p+Cz^p=0$ in $K\setminus\Q$ when $BC\neq 1$.
\end{corollary}

\section{Diophantine equations and hyperelliptic curves} 
Let $(x,y,z)$ be a solution to the Diophantine equation 
\begin{equation}\label{eq:00}
Ax^p+By^p+Cz^p=0
\end{equation}
for $p>3$ prime, $x\neq 0$, and $A$, $B$ and $C$ being pairwise coprime $p$th powerfree integers. There is a standard change of variable (Proposition 6.4.13. \cite{cohen01}) to obtain the hyperelliptic curve 

\begin{equation}\label{eq:01}
Y^2 = X^{p}+\dfrac{A^2(BC)^{p-1}}{4}
\end{equation}
given by 
\begin{equation}\label{eq:cambiodevariable} 
X=\dfrac{-BCyz}{x^2}, \quad Y=\dfrac{(-BC)^{\frac{(p-1)}{2}}(By^p-Cz^p)}{(2x^{p})}
\end{equation}
 
In other words, given a solution $(x,y,z)$ to $(\ref{eq:00})$ we can construct a point $(X,Y)$ on $(\ref{eq:01})$ given by $(\ref{eq:cambiodevariable})$. On the other hand, take the hyperelliptic curve 
$$Y^2 = X^p+(2^{p-1}A(BC)^{\frac{p-1}{2}})^2$$ 
which is obtained directly from $(\ref{eq:00})$ by multiplying it by $2^{2p}$, then $$(x,y,z)=(A^{\frac{-1}{p}}(2^{p-1}A(BC)^{\frac{p-1}{2}})^4 XY,-B^{\frac{-1}{p}}(2^{p-1}A(BC)^{\frac{p-1}{2}})^2X^pY,C^{\frac{-1}{p}}(2^{p-1}A(BC)^{\frac{p-1}{2}})^2XY^2)$$
is a $\overline{K}$-rational solution to $(\ref{eq:01})$. 

We are interested in studying what we call \textit{nontrivial solutions} in $K$ to $Ax^p+By^p+Cz^p=0$, which is any triplet $(x,y,z)$ with $x$, $y$, and $z$ in $K$ satisfying $(\ref{eq:00})$ with $xyz\neq 0$. We will prove in Lemma \ref{lem:contra}, that the only solution $(x,y,z)$ to $(\ref{eq:00})$ with $xyz=0$ is $(0,0,0)$ when $A$, $B$, $C$ are pairwise coprime with $AB\neq \pm 1$, and is $(\pm 1, 1,0)$ when $AB=\pm 1$.

\section{Arithmetic on Quadratic extensions}

We are looking for nontrivial solutions $(x,y,z)$ to $Ax^p+By^p+Cz^p=0$ over any quadratic field $K=\Q(\sqrt{d})$ with $d$ squarefree integer and, in particular, over its ring of integers $\OO_K$. It is important to remark that we do not necessarily have unique factorization in $\OO_K$, which means we have to adjust certain definitions before continuing our work.

Having this remark on our mind, let $a$, $b$, and $c$ be elements in $\OO_K$, then  we say \textit{$a$ divides $b$} if there exists a factorization in $\OO_K$ such that $b=ac$. Furthermore, we will say that $a$ and $b$ are \textit{coprime} if they do not have irreducible elements in common in any of their factorizations, and we will denote this as $\gcd(a,b)=1$. Finally, the notation $\gcd(x,y,z)=1$ means that those elements are pairwise coprime, i.e., $ \gcd(x,y)=1$, $\gcd(x,z)=1$ and $\gcd(y,z)=1$ simultaneously.

Observe that when $\OO_K$ is a unique factorization domain, the phrases ``if there exists a factorization in $\OO_K$", ``in any of their factorizations" and ``irreducible elements" are replaced by ``they can be factorized in $\OO_K$ as", ``in their factorization" and ``prime elements", respectively. Furthermore, we can apply any descent technique at any point, and it will hold up. On the other hand, when $\OO_K$ is not a unique factorization domain, we cannot apply any descent technique. The proofs presented in the following sections do not depend on any descent technique because we are not assuming $\OO_K$ is a unique factorization domain.

\begin{definition}\label{def:onlyone}
The triplet $(x,y,z)$ is a \textit{primitive solution} in $\OO_K$ to $Ax^p+By^p+Cy^p=0$ if is not trivial and $x$, $y$, and $z$ are pairwise coprime.
\end{definition}

It is not difficult to see that when $\OO_K$ is a unique factorization domain, the coefficients $A$, $B$, $C$ and any primitive solution $(x,y,z)$ naturally satisfy the conditions $\gcd(A,y,z)=1$, $\gcd(B,x,z)=1$ and $\gcd(C,x,y)=1$. On the other hand, when $\gcd(x,y)\neq 1$ and $\OO_K$ is not a unique factorization domain, we do not necessarily have that $\gcd(x,y,z)\neq 1$. In general, it only means that we can find another factorization for $Cz^p$.

\begin{proposition}\label{prop:primosrelativos00}
Let $x$ and $y$ be in $\OO_K$, then $\gcd(x,y)=1$ if and only if $\gcd(\overline{x},\overline{y})=1$.
\end{proposition}
\begin{proof}
Suppose we have $\gcd(x,y)=1$ but $\gcd(\overline{x},\overline{y})\neq 1$, there exists $\p$ an irreducible element in $\OO_K$, and $\bar{x}$, $\bar{y}$ denoting their conjugates. Then, there exists a factorization such that $\overline{x}=\p x_1$ and $\overline{y}=\p y_1$. By conjugating them again, we obtain $x=\overline{\p}\overline{x}_1$ and $y=\overline{\p}\overline{y}_1$, so $\gcd(x,y)\neq 1$, which is a contradiction.
\end{proof}
Recall that for any $x$ in $K$, we can compute its norm, which has the standard notation $N(x)$. Moreover, for any $x$ in $\OO_K$, we will have $N(x)$ in $\Z$.
\begin{proposition}\label{prop:primosrelativos01}
Let $x$ and $y$ be in $\OO_K$ such that $\gcd(N(x),N(y))=1$, then $\gcd(x,y)=1$.
\end{proposition}
\begin{proof}
Suppose we have $\gcd(x,y)\neq 1$, then there exists $\p$ irreducible in $\OO_K$, e.g.,  $x = \p x_1$ and $y=\p y_1$ for some $x_1$ and $y_1$ in $\OO_K$. Applying the norm on both numbers, we obtain $N(x)=N(\p)N(x_1)$ and $N(y)=N(\p)N(y_1)$. Therefore $\gcd(N(x),N(y))\neq 1$.
\end{proof}

Observe that $\gcd(x,y)=1$ does not necessarily imply $\gcd(N(x),N(y))=1$. For example $\gcd_{\Q(i)}(1+i,1-i)=1$ but $\gcd_\Z(N(1+i),N(1-i))=2$.

\begin{proposition}\label{prop:primosrelativos02}
Let $A$ and $B$ be integer numbers such that $\gcd_\Z(A,B)=1$, then $\gcd_K(A,B)=1$.
\end{proposition}
\begin{proof}
Since $\gcd_\Z(A,B)=1$, then $\gcd_\Z(A^2,B^2)=1$. Suppose $\gcd_K(A,B)\neq 1$ then $\gcd_\Z(N(A),N(B))\neq 1$, which is a contradiction because $N(A)=A^2$ and $N(B)=B^2$.
\end{proof}

\begin{proposition}\label{prop:primosrelativos03}
Let $y$ and $z$ be in $\OO_K$ such that $\gcd(y,z)=1$, and $A$ be an integer number. If $\gcd(A,y,z)=1$, then $\gcd(A,N(y),N(z))=1$.
\end{proposition}
\begin{proof}
Take $y$ and $z$ in $\OO_K$, and $A$ in $\Z$ such that $\gcd(A,y,z)=1$, then by Proposition $\ref{prop:primosrelativos00}$, we have that $\gcd(A,\overline{y},\overline{z})=1$. Applying the norm on $A$, $y$, and $z$, we will get $\gcd(A^2,N(y),N(z))=1$, which implies $\gcd(A,N(y),N(z))=1$.
\end{proof}

We can repeat this analysis for $\gcd(B,x,z)=1$ to get $\gcd(B,N(x),N(z))=1$ and for $\gcd(C,x,y)=1$ to get $\gcd(C,N(x),N(y))=1$. 

\subsection{The real and imaginary parts of $a+b\sqrt{d}$}

Let $a+b\sqrt{d}$ be in $K$ and recall $a$ and $b$ are in $\Q$. We say $a$ is the \textit{real} part of $a+b\sqrt{d}$, denoted by $Re(a+b\sqrt{d})=a$, and  $b$ is the \textit{imaginary} part of $a+b\sqrt{d}$, denoted by $Im(a+\sqrt{d})=b$. We are fully aware these definitions do not make sense when $d>0$, but we decided to keep them due to their practicality when writing and proving statements.

\begin{lemma}\label{lem:p2}
Let $a+b\sqrt{d}$ be in $K$ such that $a$ and $b$ are in $\Z$, then
\begin{enumerate}[$(i)$]
\item $Im\left((a+b\sqrt{d})^2\right)=0$, then either $a=0$ or $b=0$
\item $Re\left((a+b\sqrt{d})^2\right)=0$ if and only if $a=\pm b$ and $d=-1$
\end{enumerate}
\end{lemma}
\begin{proof}
Let $a+b\sqrt{d}$ be in $K$ and consider
\[
(a+b\sqrt{d})^2=a^2+b^2d+2ab\sqrt{d}
\]
\begin{enumerate}[$(i)$]
\item When $Im\left((a+b\sqrt{d})^2\right)=0$ we have $2ab=0$, then $a=0$ or $b=0$
\item When $Re\left((a+b\sqrt{d})^2\right)=0$ we have $a^2+b^2d=0$, implying $d=-(a/b)^2$. This is a contradiction unless $a=\pm b$ and $d=-1$
\end{enumerate}
\end{proof}

\begin{lemma}\label{lem:p3}
Let $a+b\sqrt{d}$ be in $K$ such that $a$ and $b$ are in $\Z$, then
\begin{enumerate}[$(i)$]
\item $Im\left((a+b\sqrt{d})^3\right)=0$ and $b\neq 0$ if and only if $a=\pm b$ and $d=-3$
\item $Re\left((a+b\sqrt{d})^3\right)=0$ and $a\neq 0$ if and only if $a=\pm 3b$ and $d=-3$
\end{enumerate}
\end{lemma}
\begin{proof}
Let $a+b\sqrt{d}$ with $a$, $b$ in $\Z$ and consider
\[
(a+b\sqrt{d})^3=\left( a^3+3ab^2d\right)+\left( 3a^2b+b^3d\right)\sqrt{d}
\]
\begin{enumerate}[$(i)$]
\item When $Im\left((a+b\sqrt{d})^3\right)=0$ and $b\neq 0$, we have that $-3a^2=b^2d$. It follows that $d=-3(a/b)^2$ is a squarefree integer, which is true if and only if $a=\pm b$ and $d=-3$. Otherwise, $b=0$ and we get a contradiction.
\item When $Re\left((a+b\sqrt{d})^3\right)=0$ and $a\neq 0$, we have that $d=-\dfrac{a^2}{3b^2}$ is a squarefree integer. This is true if and only if $a=\pm 3b$ and $d=-3$. Otherwise, $a=0$ and we get a contradiction.
\end{enumerate}
\end{proof}
Now, for any $p$ prime greater than $3$, there exists an $m$ in $\NN$ such that $p=2m+1$ and $m\geq 2$. Thus, for $(a+b\sqrt{d})^p$ we have that
\[
(a+b\sqrt{d})^p=\sum_{k=0}^m \binom p {2k} a^{p-2k}b^{2k}d^{k}+\left(\sum_{k=0}^m \binom p {2k+1} a^{p-(2k+1)}b^{2k+1}d^{k}\right)\sqrt{d}
\]
and
\begin{align*}
Re\left((a+b\sqrt{d})^p\right)&=a\left(\sum_{k=0}^m \binom p {2k} a^{p-(2k+1)}b^{2k}d^{k}\right)\\
Im\left((a+b\sqrt{d})^p\right)&=b\left(\sum_{k=0}^m \binom p {2k+1} a^{p-(2k+1)}b^{2k}d^{k}\right)
\end{align*}

Observe that $p$ divides $\binom p j$ for $j\in \lbrace 1,\ldots,p-1\rbrace$. 

\begin{lemma}\label{lem:imre_i}
Let $a+b\sqrt{d}$ be in $K$ with $a$, $b$ in $\Z$, $d$ a squarefree integer and $p>3$ a prime number such that $\gcd(p,d)=1$, then
\begin{enumerate}[$(i)$]
\item $Im\left((a+b\sqrt{d})^p\right)=0$ if and only if $b=0$
\item $Re\left((a+b\sqrt{d})^p\right)=0$ if and only if $a=0$
\end{enumerate} 
\end{lemma}
\begin{proof}
Let $a+b\sqrt{d}$ be in $K$ with $a$, $b$ in $\Z$.
\begin{enumerate}[$(i)$]
\item When $Im\left((a+b\sqrt{d})^p\right)=0$ with $b\neq 0$ and $p=2m+1$ for some $m\in\NN$, then
\begin{align*}
0&=\sum_{k=0}^m \binom p {2k+1} a^{p-(2k+1)}b^{2k+1}d^{k}\\
&=b\left(b^{p-1}d^m+pa^{p-1}+\sum_{k=1}^{m-1} \binom p {2k+1} a^{p-(2k+1)}b^{2k}d^{k}\right)
\end{align*}
Since $b\neq 0$, we have
\begin{align*}
-b^{p-1} d^m&=pa^{p-1}+\sum_{k=1}^{m-1} \binom p {2k+1} a^{p-(2k+1)}b^{2k}d^{k}
\end{align*}
which implies $p\vert b^{p-1}$ and thus $p\vert b$. Then, for each term of sum, we have that
\[
p^{2k+1}\mbox{ divides } \binom p {2k+1} a^{p-(2k+1)}b^{2k}d^{k}
\]
Thus, $p^3 \vert a^{p-1}$ and then $p\vert a$. Due to this, for each term of sum, we have that
\[
p^{p-(2k+1)+2k+1}\mbox{ divides } \binom p {2k+1} a^{p-(2k+1)}b^{2k}d^{k}
\]
Since $p^{p-(2k+1)+2k+1}=p^{p}$, then $p^p \vert -b^{p-1}d^m$ and thus $p^2\vert b$.
\\
Now, suppose $n$ is the maximal natural number such that $p^n \vert b$ and $p^{n-1}\vert a$. In this way, for each term of sum, we have that
\[
p^{(n-1)\left[ p-(2k+1)\right] +2nk+1}\mbox{ divides } \binom p {2k+1} a^{p-(2k+1)}b^{2k}d^{k}
\]
Since $p^{(n-1)\left[p-(2k+1)\right]+2nk+1}=p^{(p-1)(n-1)+2k+1}$, then $p^{(p-1)(n-1)+3}$ divides the sum and $p^{n(p-1)}=p^{(n-1)(p-1)+p-1}\vert b^{p-1}$, then $p^{(n-1)(p-1)+3} \vert a^{p-1}$ and thus $p^n\vert a$. Due to this, for each term of sum, we have that
\[
p^{n\left[p-(2k+1)\right]+2nk+1}\mbox{ divides } \binom p {2k+1} a^{p-(2k+1)}b^{2k}d^{k}
\]
Since $p^{n\left[p-(2k+1)\right]+2nk+1}=p^{n(p-1)+1}$ divides the sum, we then get $p^{n(p-1)+1} \vert -b^{p-1}d^m$ and thus $p^{n+1}\vert b$, which contradicts the maximality of $n$.
\item When $Re\left( (a+b\sqrt{d})^p\right)=0$, we have
\begin{align*}
0&=\sum_{k=0}^m \binom p {2k} a^{p-2k}b^{2k}d^{k}\\
&=a\left(a^{p-1}+pb^{p-1}d^m+\sum_{k=1}^{m-1} \binom p {2k} a^{p-(2k+1)}b^{2k}d^{k}\right)\\
\end{align*}
Suppose $a\neq 0$, then
\[
-a^{p-1}=pb^{p-1}d^m+\sum_{k=1}^{m-1} \binom p {2k} a^{p-(2k+1)}b^{2k}d^{k}
\]
implying $p\vert a^{p-1}$ and thus $p\vert a$. Then, for each term of sum, we have that
\[
p^{p-2k}\mbox{ divides } \binom p {2k} a^{p-(2k+1)}b^{2k}d^{k}
\]
which implies $p^{p-2}\vert pb^{p-1}d^m$, so $p\vert b$. Hence, for each term of sum, we have that
\[
p^{p}\mbox{ divides } \binom p {2k} a^{p-(2k+1)}b^{2k}d^{k}
\]
meaning $p^p\vert pb^{p-1}d^m$, thus $p^p\vert a^{p-1}$, so $p^2\vert a$.
\\
Now, let $n$ be the maximal natural number such that $p^n\vert a$ and $p^{n-1}\vert b$, then
\[
p^{n\left[ p-(2k+1)\right] +2k(n-1)+1}\mbox{ divides } \binom p {2k} a^{p-(2k+1)}b^{2k}d^{k}
\]
Since $p^{n\left[ p-(2k+1)\right] +2k(n-1)+1}=p^{(n-1)(p-1)+p-2k}$, then $p^{(n-1)(p-1)+3}$ divides the sum and $p^{n(p-1)}=p^{(n-1)(p-1)+p-1}\vert a^{p-1}$, we get $p^{(n-1)(p-1)+3} \vert pb^{p-1}d^m$, so $p^n\vert b$. Due to this, for each term of sum, we have that
\[
p^{n\left[p-(2k+1)\right]+2nk+1}\mbox{ divides } \binom p {2k} a^{p-(2k+1)}b^{2k}d^{k}
\]
Since $p^{n\left[p-(2k+1)\right]+2nk+1}=p^{n(p-1)+1}$ divides the sum and $p^{n(p-1)+1} \vert pb^{p-1}d^m$, then $p^{n(p-1)+1}\vert a^{p-1}$ and so $p^{n+1}\vert a$, which contradicts the maximality of $n$.
\end{enumerate}
\end{proof}

\begin{lemma}\label{lem:imre_ii}
Let $a+b\sqrt{d}$ be in $K$ with $a$, $b$ in $\Z$, $d$ a squarefree integer and $p>3$ a prime number such that $\gcd(p,d)=p$, then 
\begin{enumerate}[$(i)$]
\item $Im\left((a+b\sqrt{d})^p\right)=0$ if and only if $b=0$
\item $Re\left((a+b\sqrt{d})^p\right)=0$ if and only if $a=0$
\end{enumerate}
\end{lemma}
\begin{proof}
Let $a+b\sqrt{d}$ be in $K$ with $a$, $b$ in $\Z$
\begin{enumerate}[$(i)$]
\item When $Im\left((a+b\sqrt{d})^p\right)=0$ with $b\neq 0$ and $p=2m+1$ for some $m\in\NN$, then
\[
b\left(b^{p-1}d^m+pa^{p-1}+\sum_{k=1}^{m-1} \binom p {2k+1} a^{p-(2k+1)}b^{2k}d^{k}\right)=0
\]
Since $b\neq 0$, we have that
\begin{align*}
-b^{p-1} d^m&=pa^{p-1}+\sum_{k=1}^{m-1} \binom p {2k+1} a^{p-(2k+1)}b^{2k}d^{k}
\end{align*}
and observe $p$ divides $\binom p {2k+1}$ for $k\in \lbrace 1,2,\ldots,m-1\rbrace$. Then, for each term of sum, we have that
\[
p^{k+1}\mbox{ divides } \binom p {2k+1} a^{p-(2k+1)}b^{2k}d^{k}
\]
Moreover, we have $p^m\vert b^{p-1}d^m$, thus $p^2 \vert pa^{p-1}$, so $p\vert a$. From here on, the proof is followed as we did in $(i)$ at Lemma \ref{lem:imre_i}.

\item When $Re\left((a+b\sqrt{d})^p\right)=0$, we have
\[
a\left(a^{p-1}+pb^{p-1}d^m+\sum_{k=1}^{m-1} \binom p {2k} a^{p-(2k+1)}b^{2k}d^{k}\right)=0
\]
Suppose $a\neq 0$, then
\begin{align*}
-a^{p-1}&=pb^{p-1}d^m+\sum_{k=1}^{m-1} \binom p {2k} a^{p-(2k+1)}b^{2k}d^{k}
\end{align*}
and observe $p$ divides $\binom p {2k}$ for  $k\in \lbrace 1,2,\ldots,m-1\rbrace$. Hence, for each term of sum, we have that
\[
p^{k+1}\mbox{ divides } \binom p {2k} a^{p-(2k+1)}b^{2k}d^{k}
\]
and observe  $p^{m+1}\vert pb^{p-1}d^m$. Thus $p^2 \vert a^{p-1}$, so $p\vert a$. From here on, the proof is followed as we did in $(ii)$ at Lemma \ref{lem:imre_i}.
\end{enumerate}
\end{proof}

\begin{theorem}\label{thm:imaginarypartzero}
Let $a+b\sqrt{d}$ be in $K$, $d$ be a squarefree integer, and $p$ be a prime number in $\Z$, then
\begin{enumerate}[$(a)$]
\item When $p=2$,
\begin{enumerate}[$(i)$]
\item $Im\left((a+b\sqrt{d})^2\right)=0$, then either $a=0$ or $b=0$
\item $Re\left((a+b\sqrt{d})^2\right)=0$ if and only if $a=\pm b$ and $d=-1$
\end{enumerate}
\item When $p=3$,
\begin{enumerate}[$(i)$]
\item $Im\left((a+b\sqrt{d})^3\right)=0$ and $b\neq 0$ if and only if $a=\pm b$ and $d=-3$
\item $Re\left((a+b\sqrt{d})^3\right)=0$ and $a\neq 0$ if and only if $a=\pm 3b$ and $d=-3$
\end{enumerate}
\item Otherwise,
\begin{enumerate}[$(i)$]
\item $Im\left((a+b\sqrt{d})^p\right)=0$ if and only if $b=0$
\item $Re\left((a+b\sqrt{d})^p\right)=0$ if and only if $a=0$
\end{enumerate}
\end{enumerate}
\end{theorem}
\begin{proof}
Let $a+b\sqrt{d}$ be in $K$. Since $a$ and $b$ are in $\Q$, then we can rake $a=a_1/a_2$ and $b=b_1/b_2$ with $a_1,a_2,b_1,b_2\in\Z$ such that $\gcd(a_1,a_2)=\gcd(b_1,b_2)=1$, $a_2b_2\neq 0$. Thus
\begin{align*}
(a+b\sqrt{d})^p&=\dfrac{(a_1b_2+a_2b_1\sqrt{d})^p}{(a_2b_2)^p}
\end{align*}
When $Im\left((a+b\sqrt{d})^p\right)=0$, there exists $q$ in $\Q$ such that
\[
(a_1b_2+a_2b_1\sqrt{d})^p=q(a_2b_2)^p\in\Z
\]
Similarly, when $Re\left((a+b\sqrt{d})^p\right)=0$, there exists $q$ in $\Q$ such that
\[
(a_1b_2+a_2b_1\sqrt{d})^p=q(a_2b_2)^p\sqrt{d}
\]
with $q(a_2b_2)^p\in\Z$. By applying lemmata \ref{lem:p2}, \ref{lem:p3}, \ref{lem:imre_i}, and \ref{lem:imre_ii} to $a_1b_2+a_2b_1\sqrt{d}$ we conclude the proof.
\end{proof}

\section{Points and primitive solutions} 
Let $Y^2 = X^p+\alpha$ be a hyperelliptic curve over $K$ for $p>3$ prime, $\alpha$ a rational number, and consider a point $(X,Y)$ on such a hyperelliptic curve. We can classify all of these points depending on whether $X$ is in $\Q$ or $K\setminus\Q$. 

\begin{proposition}\label{prop:ifandonlyif}
Let $(X,Y)$ be a point on $Y^2=X^p+\alpha$ with $\alpha$ in $\Q$. Then 
\begin{enumerate}[$(a)$]
\item $X$ is in $\Q$ if and only if either $Im(Y)=0$ or $Re(Y)=0$
\item $X=a+b\sqrt{d}$ with $b\neq 0$ if and only if $Y=m+n\sqrt{d}$ with $mn\neq 0$
\end{enumerate}
\end{proposition}
\begin{proof}
~~
\begin{enumerate}[$(a)$]
\item Suppose $X$ is in $\Q$, then $X^p+\alpha$ is also in $\Q$. Thus, $Y^2$ has to be in $\Q$. So, by Theorem $\ref{thm:imaginarypartzero}$, we have either $Im(Y)=0$ or $Re(Y)=0$. On the other hand, suppose $Im(Y)=0$ or $Re(Y)=0$, then $Y^2$ and $Y^2-\alpha$ are also in $\Q$, thus $X^p$ is in $\Q$. Finally, by Theorem $\ref{thm:imaginarypartzero}$, we conclude that $X$ is in $\Q$.
\item By the previous case, we have that $X$ is in $K\setminus\Q$ if and only if $Im(Y)\neq 0$ and $Re(Y)\neq 0$. In other words, we have $X=a+b\sqrt{d}$ with $b\neq 0$ if and only if $Y=m+n\sqrt{d}$ with $mn\neq 0$.
\end{enumerate}
\end{proof}

\subsection{Types of solutions on our Diophantine equation}

We now analyse the solutions we could have on our Diophantine equation and determine whether a point on our hyperelliptic curve comes from a primitive solution.

\begin{proposition}\label{thm:multiplosracionales}
Let $(x,y,z)$ be a nontrivial solution in $K$ to $Ax^p+By^p+Cz^p=0$ such that $y = \gamma x$ with $\gamma\in\Q$. Then, there exists a primitive solution $(\pm \delta_2,\gamma_1,\pm\delta_1)$ in $\Z$ to $Ax^p+By^p+Cz^p=0$.
\end{proposition}
\begin{proof}
Let $(x,y,z)$ be a nontrivial solution to $(\ref{eq:00})$ in $K$ with $y=\gamma x$ and $\gamma\in\Q$. Observe that $\gamma=\gamma_1/\gamma_2$, $\gamma_2\neq 0$ and $\gcd(\gamma_1,\gamma_2)=1$, which means $(x,\gamma x,z)$ satisfies
\begin{align*}
Ax^p+B\gamma^p x^p+Cz^p&=0\\
(A+B\gamma^p)+C\dfrac{z^p}{x^p}&=0\\
\left(\dfrac{z}{x}\right)^p&=\dfrac{A+\gamma^p}{-C}\in\Q 
\end{align*}
By Theorem \ref{thm:imaginarypartzero}, we have that $\dfrac{z}{y}=\delta$ for some $\delta\in\Q$ with $\delta=\delta_1/\delta_2$, $\delta_2\neq 0$ and $\gcd(\delta_1,\delta_2)=1$, so $z=\delta x$ and we have that $(x,\gamma x,\delta x)$ satisfies
\begin{align*}
Ax^p+B\gamma^p x^p+C \delta^p x^p&=0\\
A+B\gamma^p+C\delta^p &=0\\
A\gamma_2^p\delta_2^p+B\gamma_1^p\delta_2^p+C\gamma_2^p\delta_1^p&=0
\end{align*}
Thus, $(\gamma_2\delta_2, \gamma_1\delta_2,\gamma_2\delta_1)$ is a nontrivial solution in $\Z$ to $(\ref{eq:00})$. 
\\
Now, for each prime integer $q$ such that $q^\alpha$ is the maximum power of $q$ dividing $\gamma_2$, we have that $q^{\alpha p}\vert B\gamma_1^p\delta_2^p$. Since $B$ is a $p$th powerfree integer and $\gcd(\gamma_1,\gamma_2)=1$, then $q^{(\alpha-1)p+1}\vert \delta_2^p$ and thus $q^\alpha\vert \delta_2$. On the other hand, let $q^\beta$ be the maximum power of $q$ dividing $\delta_2$, then $q^{\beta p}\vert C\gamma_2^p\delta_1^p$. Since $C$ is a $p$th powerfree integer and $\gcd(\delta_1,\delta_2)=1$, then $q^{(\beta-1)p+1}\vert \gamma_2^p$ and thus $q^\beta\vert \gamma_2$. Due to the maximality of $\alpha$ and $\beta$ we get that $\alpha=\beta$. In this way, we can conclude $\gamma_2=\pm \delta_2$. Therefore $(\pm\delta_2,\gamma_1,\pm\delta_1)$ is a primitive solution in $\Z$ for  $(\ref{eq:00})$.
\end{proof}

\begin{remark}
Particularly, when $(x,y,z)$ is a nontrivial solution in $\OO_K$ to $(\ref{eq:00})$, we will have that
\begin{align*}
y &= \dfrac{\gamma_1}{\gamma_2}(x_1+x_2\omega_d)\\
&=\dfrac{\gamma_1 x_1}{\gamma_2}+\dfrac{\gamma_1 x_2}{\gamma_2}\omega_d
\end{align*}
is still in $\OO_K$ for some $x_1,x_2\in\Z$. Since $\gcd(\gamma_1,\gamma_2)=1$, then $\gamma_2\vert x_1$ and $\gamma_2\vert x_2$, so we define
\begin{align*}
x_{\gamma} &= \dfrac{x_1}{\gamma_2}+\dfrac{x_2}{\gamma_2}\omega_d
\end{align*}
which again remains in $\OO_K$. Similarly, since $\gcd(\delta_1,\delta_2)=1$, we can define
\begin{align*}
x_\delta &=  \dfrac{x_1}{\delta_2}+\dfrac{x_2}{\delta_2}\omega_d
\end{align*}
where $x_\delta$ is in $\OO_K$. These two equations imply that $x=\gamma_2 x_\gamma$ and $x=\delta_2 x_\delta$, but recall $\gamma_2=\pm\delta_2$, then $(\pm \delta_2x_\gamma,\gamma_1x_\gamma,\pm \delta_1x_\gamma)$ is a nontrivial solution in $\OO_K$, and in order to be a primitive solution in $\OO_K\backslash \Z$, we have to have that $x_\gamma$ is a unit in $\OO_K$ different from $\pm 1$.

Furthermore, we know that for $d<0$, we get three cases to analyse:
\begin{enumerate}[$(a)$]
\item For $d=-1$ we get $x_\gamma\in\lbrace \pm i\rbrace$
\item For $d=-3$ we get $x_\gamma\in\lbrace \pm \omega,\pm\omega^2\rbrace$, with $\omega$ a cubic root of the unit.
\item For $d\neq -1,-3$ we get no other but $x_\gamma\in\lbrace \pm 1\rbrace$
\end{enumerate}
Finally, for $d>0$, we get infinitely many units defined as powers of a fundamental unit $u$, so $x_\gamma=u^n$ for any $n\in\NN$.
\end{remark}

\begin{lemma}\label{lem:contra}
Let $(x,y,z)$ be a solution in $K$ to $Ax^p+By^p+Cz^p=0$ such that $xyz=0$, then $(x,y,z)$ is $(0,0,0)$ when $A$, $B$, $C$ are pairwise coprime and $AB\neq \pm 1$, and $(x,y,z)$ is $(\pm 1,1,0)$ when $AB=\pm 1$.
\end{lemma}
\begin{proof}
Let $(x,y,z)$ be a solution in $K$ to $(\ref{eq:00})$ such that $xyz=0$,  then we have three cases to analyse:
\begin{enumerate}[$(a)$]
\item When all $x=y=z=0$, then we have the trivial solution 
\item When we have $(x,y,0)$ with $xy\neq 0$, then
\begin{align*}
Ax^p+By^p&=0\\
\dfrac{A}{B}&=-\dfrac{y^p}{x^p}
\end{align*}
Applying the norm on both sides of the equation give us
$$\dfrac{A^2}{B^2}=\dfrac{N(y)^p}{N(x)^p}$$ meaning $\dfrac{N(y)}{N(x)}=\sqrt[p]{\dfrac{A^2}{B^2}}$ lies in $\Q$. This is a clear contradiction since both $A$ and $B$ are $p$th powerfree integers unless $AB=\pm 1$. In particular, when $AB=\pm 1$ we will have that $x=\pm y$ and the solution is the triplet $(\pm y, y,0)$, which is not primitive, but reducible to $(\pm 1, 1, 0)$. This solution is a trivial one.
\item When we have $(x,0,0)$, then $Ax^p=0$ implies $x=0$
\end{enumerate}
\end{proof}

\subsection{Points on Hyperelliptic curves coming from primitive solutions}

Let $(x,y,z)$ be a nontrivial solution in $K$ to $(\ref{eq:00})$, then take $(X,Y)$ to be a point on $(\ref{eq:01})$ given by $(\ref{eq:cambiodevariable})$. In this section, we are determining the necessary conditions for $(x,y,z)$ to be a primitive solution to $(\ref{eq:00})$ depending on whether $Y$ is in $\Q$ or $K\setminus\Q$.

\begin{proposition}\label{thm:finalthm01}
Let $(X,Y)$ be a point on $Y^2 =X^p+(A^2(BC)^{p-1})/4$ where $Y=m+n\sqrt{d}$ with $m,n\in\Q$ given by 
$$m+n\sqrt{d}=\dfrac{(-BC)^{\frac{p-1}{2}}(By^p-Cz^p)}{2x^p}$$
for any solution $(x,y,z)$ in $K$ to $Ax^p+By^p+Cz^p=0$, then
\begin{align}\label{eq:Ymn01}
\dfrac{2m}{(-BC)^{\frac{p-1}{2}}}=B\dfrac{y^p}{x^p}-C\dfrac{\overline{z}^p}{\overline{x}^p}, \quad \dfrac{2n}{(-BC)^{\frac{p-1}{2}}}\sqrt{d}=-C\dfrac{z^p}{x^p}+C\dfrac{\overline{z}^p}{\overline{x}^p}
\end{align}
\end{proposition}
\begin{proof}
Let be $$m+n\sqrt{d}=\dfrac{(-BC)^{\frac{p-1}{2}}(By^p-Cz^p)}{2x^p}$$ where $m,n\in\Q$, then
\begin{align}\label{eq:ft01}
\dfrac{2m}{(-BC)^{\frac{p-1}{2}}}+\dfrac{2n}{(-BC)^{\frac{p-1}{2}}}\sqrt{d}=B\dfrac{y^p}{x^p}-C\dfrac{z^p}{x^p}
\end{align}
Since this expression lies in $K$, we can conjugate it to obtain
\begin{align}\label{eq:2mbc+2nbc}
\dfrac{2m}{(-BC)^{\frac{p-1}{2}}}-\dfrac{2n}{(-BC)^{\frac{p-1}{2}}}\sqrt{d}=B\dfrac{\overline{y}^p}{\overline{x}^p}-C\dfrac{\overline{z}^p}{\overline{x}^p}
\end{align}
Moreover, we can multiply $(\ref{eq:2mbc+2nbc})$ by $-1$ to get 
\begin{align}\label{eq:ft02}
-\dfrac{2m}{(-BC)^{\frac{p-1}{2}}}+\dfrac{2n}{(-BC)^{\frac{p-1}{2}}}\sqrt{d}=-B\dfrac{\overline{y}^p}{\overline{x}^p}+C\dfrac{\overline{z}^p}{\overline{x}^p}
\end{align}
Recall $(x,y,z)$ is a solution to $(\ref{eq:00})$, then
\begin{align}\label{eq:ft03}
A=-B\dfrac{y^p}{x^p}-C\dfrac{z^p}{x^p}
\end{align}
and
\begin{align}\label{eq:ft04}
A=-B\dfrac{\overline{y}^p}{\overline{x}^p}-C\dfrac{\overline{z}^p}{\overline{x}^p}
\end{align}
Adding equation $(\ref{eq:ft01})$ to $(\ref{eq:ft03})$, we get
\begin{align}\label{eq:ft05}
A+\dfrac{2m}{(-BC)^{\frac{p-1}{2}}}+\dfrac{2n}{(-BC)^{\frac{p-1}{2}}}\sqrt{d}=-2C\dfrac{z^p}{x^p}
\end{align}
and adding equation $(\ref{eq:ft02})$ to $(\ref{eq:ft04})$, we get
\begin{align}\label{eq:ft06}
A-\dfrac{2m}{(-BC)^{\frac{p-1}{2}}}+\dfrac{2n}{(-BC)^{\frac{p-1}{2}}}\sqrt{d}=-2B\dfrac{\overline{y}^p}{\overline{x}^p}
\end{align}
Furthermore, when adding $(\ref{eq:ft05})$ to $(\ref{eq:ft06})$, we get that
\begin{align*}
2A+2\dfrac{2n}{(-BC)^{\frac{p-1}{2}}}\sqrt{d}&=-2C\dfrac{z^p}{x^p}-2B\dfrac{\overline{y}^p}{\overline{x}^p}\\\nonumber
A+\dfrac{2n}{(-BC)^{\frac{p-1}{2}}}\sqrt{d}&=-C\dfrac{z^p}{x^p}-B\dfrac{\overline{y}^p}{\overline{x}^p}\\\nonumber
\dfrac{2n}{(-BC)^{\frac{p-1}{2}}}\sqrt{d}&=-C\dfrac{z^p}{x^p}-B\dfrac{\overline{y}^p}{\overline{x}^p}-A
\end{align*}
and by $(\ref{eq:ft04})$, we get that
\begin{align}\label{eq:2nbc}
\dfrac{2n}{(-BC)^{\frac{p-1}{2}}}\sqrt{d}=-C\dfrac{z^p}{x^p}+C\dfrac{\overline{z}^p}{\overline{x}^p}
\end{align}
Finally, we substitute $(\ref{eq:2nbc})$ into $(\ref{eq:ft01})$ to get
\begin{align}\label{eq:2mbc}
\dfrac{2m}{(-BC)^{\frac{p-1}{2}}}=B\dfrac{y^p}{x^p}-C\dfrac{\overline{z}^p}{\overline{x}^p}
\end{align}
\end{proof}

\begin{theorem}\label{thm:finalthm00}
Let $(X,Y)$ be a point on $Y^2 =X^p+(A^2(BC)^{p-1})/4$ where $Y=m+n\sqrt{d}$ is coming from a nontrivial solution $(x,y,z)$ in $K$ to $Ax^p+By^p+Cz^p=0$, then $mn=0$.
\end{theorem}
\begin{proof}
Suppose there is a nontrivial solution $(x,y,z)$ in $K$ to $(\ref{eq:00})$ such that we have a point on $Y^2 =X^p+(A^2(BC)^{p-1})/4$, where $Y=m+n\sqrt{d}$ and $mn\neq 0$. Take equation $(\ref{eq:2mbc})$ and multiply it by $n$ to get
\begin{align}\label{eq:13timesn}
\dfrac{2mn}{(-BC)^{\frac{p-1}{2}}}=nB\dfrac{y^p}{x^p}-nC\dfrac{\overline{z}^p}{\overline{x}^p}
\end{align}
Furthermore, take equation $(\ref{eq:2nbc})$ and multiply it by $m$ to get 
\begin{align}\label{eq:12timesm}
\dfrac{2mn}{(-BC)^{\frac{p-1}{2}}}\sqrt{d}=-mC\dfrac{z^p}{x^p}+mC\dfrac{\overline{z}^p}{\overline{x}^p}
\end{align}
Then substitute $(\ref{eq:13timesn})$ into $(\ref{eq:12timesm})$ to obtain
\begin{align}\label{eq:ft201}
\left(nB\dfrac{y^p}{x^p}-nC\dfrac{\overline{z}^p}{\overline{x}^p} \right)\sqrt{d}&=-mC\dfrac{z^p}{x^p}+mC\dfrac{\overline{z}^p}{\overline{x}^p}\\\nonumber
nB\dfrac{y^p}{x^p}\sqrt{d}-nC\dfrac{\overline{z}^p}{\overline{x}^p}\sqrt{d}&=-mC\dfrac{z^p}{x^p}+mC\dfrac{\overline{z}^p}{\overline{x}^p}\\\nonumber
mC\dfrac{z^p}{x^p}+nB\dfrac{y^p}{x^p}\sqrt{d}&=mC\dfrac{\overline{z}^p}{\overline{x}^p}+nC\dfrac{\overline{z}^p}{\overline{x}^p}\sqrt{d}
\end{align}
Recall $Ax^p+By^p+Cz^p=0$, so 
\begin{align}\label{eq:ft202}
B\dfrac{y^p}{x^p}=-A-C\dfrac{z^p}{x^p}
\end{align}
and
\begin{align}\label{eq:ft203}
C\dfrac{z^p}{x^p}=-A-B\dfrac{y^p}{x^p}
\end{align}
Now, substitute $(\ref{eq:ft203})$ into $(\ref{eq:ft201})$ to get
\begin{align*}
mC\dfrac{\overline{z}^p}{\overline{x}^p}+nC\dfrac{\overline{z}^p}{\overline{x}^p}\sqrt{d}
&= mC\dfrac{z^p}{x^p}+nB\dfrac{y^p}{x^p}\sqrt{d}\\\nonumber
&= m\left(-A-B\dfrac{y^p}{x^p}\right)+nB\dfrac{y^p}{x^p}\sqrt{d}\\\nonumber
&=-Am-mB\dfrac{y^p}{x^p}+nB\dfrac{y^p}{x^p}\sqrt{d}
\end{align*}
thus
\begin{align*}
-Am 
&= mC\dfrac{\overline{z}^p}{\overline{x}^p}+nC\dfrac{\overline{z}^p}{\overline{x}^p}\sqrt{d}+ mB\dfrac{y^p}{x^p}-nB\dfrac{y^p}{x^p}\sqrt{d}\\
&= C\dfrac{\overline{z}^p}{\overline{x}^p}(m+n\sqrt{d})-B\dfrac{y^p}{x^p}(m-n\sqrt{d})
\end{align*}
Multiplying $-Am$ by $-N(x^p)$ we get
\begin{align}\label{eq:noname01}
AmN(x^p)&=-Cx^p\overline{z}^p(m+n\sqrt{d})+B\overline{x}^py^p(m-n\sqrt{d})
\end{align}
On the other hand, substitute $(\ref{eq:ft202})$ into $(\ref{eq:ft201})$ to get
\begin{align*}
mC\dfrac{\overline{z}^p}{\overline{x}^p}+nC\dfrac{\overline{z}^p}{\overline{x}^p}\sqrt{d}
&= mC\dfrac{z^p}{x^p}+nB\dfrac{y^p}{x^p}\sqrt{d}\\
&= mC\dfrac{z^p}{x^p}+n\left(-A-C\dfrac{z^p}{x^p} \right)\sqrt{d}\\
&= mC\dfrac{z^p}{x^p}-nA\sqrt{d}-nC\dfrac{z^p}{x^p}\sqrt{d}
\end{align*}
thus
\begin{align*}
-An\sqrt{d}
&= mC\dfrac{\overline{z}^p}{\overline{x}^p}+nC\dfrac{\overline{z}^p}{\overline{x}^p}\sqrt{d} -mC\dfrac{z^p}{x^p}+nC\dfrac{z^p}{x^p}\\
&= C\dfrac{\overline{z}^p}{\overline{x}^p}(m+n\sqrt{d})-C\dfrac{z^p}{x^p}(m-n\sqrt{d})
\end{align*}
Multiplying $-An\sqrt{d}$ by $N(x^p)$ we get
\begin{align}\label{eq:noname02}
-AnN(x^p)\sqrt{d}&=Cx^p\overline{z}^p(m+n\sqrt{d})-C\overline{x}^pz^p(m-n\sqrt{d})
\end{align}
Adding equations $(\ref{eq:noname01})$ and $(\ref{eq:noname02})$ together, we get
\begin{align*}
AmN(x^p)-AnN(x^p)\sqrt{d}&= B\overline{x}^py^p(m-n\sqrt{d})-C\overline{x}^pz^p(m-n\sqrt{d})\\
AN(x^p)(m-n\sqrt{d})&=(B\overline{x}^py^p-C\overline{x}^pz^p)(m-n\sqrt{d})\\
Ax^p\overline{x}^p&=B\overline{x}^py^p-C\overline{x}^pz^p\\
Ax^p&=By^p-Cz^p
\end{align*}
which implies $(x,y,z)$ satisfies both equations
\[
Ax^p+By^p+Cz^p=0, \quad Ax^p-By^p+Cz^p=0
\]
so $y=0$. Therefore, by Lemma $\ref{lem:contra}$, we have that $(x,y,z)$ is a trivial solution, which is a contradiction.
\end{proof}

At this point, let $(x,y,z)$ be a nontrivial solution in $K$ to $(\ref{eq:00})$ and take $(X,Y)$ to be a point on $(\ref{eq:01})$ given by $(\ref{eq:cambiodevariable})$. Then, we have for $Y$ that either $Re(Y)=0$ or $Im(Y)=0$. We need to analyse first what happens when $Re(Y)=0$.

\begin{proposition}
Let $(X,Y)$ be a point on $Y^2 =X^p+(A^2(BC)^{p-1})/4$ with $Y=n\sqrt{d}$ coming from a nontrivial solution $(x,y,z)$ in $K$ to $Ax^p+By^p+Cz^p=0$, then $X\in\Q$, $BC=\pm 1$ and $(x,y,z)=(x,u\overline{z},z)$ where $u$ is a unit.
\end{proposition}
\begin{proof}
Suppose $m=0$ in $(\ref{eq:Ymn01})$ for a nontrivial solution $(x,y,z)$ in $K$ to $(\ref{eq:00})$, then
\[
B\dfrac{y^p}{x^p}-C\dfrac{\overline{z}^p}{\overline{x}^p}=0
\]
thus
\begin{align*}
B\overline{x}^py^p &=Cx^p\overline{z}^p\\
\dfrac{B}{C}&=\dfrac{x^p\overline{z}^p}{\overline{x}^py^p}
\end{align*}
Applying the norm on both sides of the equation give us
\begin{align*}
\dfrac{B^2}{C^2}&=\dfrac{N(z)^{p}}{N(y)^{p}}
\end{align*}
so $\sqrt[p]{\left(\dfrac{B^2}{C^2}\right)}=\dfrac{N(z)}{N(y)}\in\Q$, which is a contradiction unless $BC=\pm 1$.\\ 
Now, let be $BC=\pm 1$ and take $\dfrac{y}{\overline{z}}=u$ in $K$. Then $u$ is a unit because $N(y)=N(z)$. So $y = u\overline{z}$ and therefore $(x,y,z)=(x,u\overline{z},z)$ is the solution.
\end{proof}

Now, we analyse what happens when $Im(Y)=0$.

\begin{proposition}
Let $(X,Y)$ be a point on $Y^2 =X^p+(A^2(BC)^{p-1})/4$ with $Y$ in $\Q$, then there exists a primitive solution in $\Z$ to $Ax^p+By^p+Cz^p=0$.
\end{proposition}
\begin{proof}
Suppose $n=0$ in $(\ref{eq:Ymn01})$ for a nontrivial solution $(x,y,z)$ in $K$ to $(\ref{eq:00})$, then $\dfrac{z^p}{x^p}=\dfrac{\overline{z}^p}{\overline{x}^p}$ is a rational number. Thus, by Theorem \ref{thm:imaginarypartzero}, we have that
$\dfrac{z}{x}=\gamma$ with $\gamma\in\Q$, i.e., $z=\gamma x$. Finally, by Proposition $\ref{thm:multiplosracionales}$ there exists a primitive solution in $\Z$ to $(\ref{eq:00})$.
\end{proof}

In particular, when $Y=0$, we have the following result.

\begin{proposition}\label{prop:patch}
Let $(X,Y)$ be a point on $Y^2 =X^p+(A^2(BC)^{p-1})/4$ with $Y=0$ coming from a nontrivial solution $(x,y,z)$ in $K$ to $Ax^p+By^p+Cz^p=0$, then $A=\pm 2$, $BC=\pm 1$ and $(x,y,z)=(\pm 1,\pm 1,1)$.
\end{proposition} 
\begin{proof}
Suppose $Y=0$ for a nontrivial solution $(x,y,z)$ in $K$ to $(\ref{eq:00})$, then by $(\ref{eq:cambiodevariable})$ we have that $Cz^p=By^p$,
implying that
\begin{align*}
\dfrac{C}{B}&=\dfrac{y^p}{z^p}
\end{align*}
Applying the norm on both sides of the equation give us
\begin{align*}
\dfrac{C^2}{B^2}&=\left(\dfrac{N(y)}{N(z)}\right)^p\\
\sqrt[p]{\left(\dfrac{C}{B}\right)^2}&=\dfrac{N(y)}{N(z)}\in\Q
\end{align*}
This is a contradiction since both $B$ and $C$ are $p$th powerfree and $p$ is odd unless $BC=\pm 1$. When $BC=\pm 1$, we will have $\dfrac{y^p}{z^p}=\pm 1 $, implying that $y = \pm z$, and so
\[
\dfrac{z^p}{x^p}=\dfrac{\pm A }{2}
\]
Since $A$ is a $p$th powerfree integer we will have that $\sqrt[p]{\dfrac{\pm A}{2}}$ lies in $\Q$ if and only if $A=\pm 2$. This means $\dfrac{x^p}{z^p}=\pm 1$ and, in this way $x=\pm z$ and $y=\pm z$. Therefore $(x,y,z)=(\pm1,\pm 1,1)$.
\end{proof}
In particular, when $X=0$, we would have that either $y$ or $z$ is $0$, and by Lemma \ref{lem:contra}, we will have that $(x,y,z)$ is the trivial solution $(0,0,0)$, which is a contradiction.

\begin{corollary}
There are no solutions $(x,y,z)$ to $Ax^p+By^p+Cz^p=0$ in $K\setminus\Q$ when $BC\neq \pm 1$.
\end{corollary}

\section{Particular case $ABC=\pm 1$}
Let $(x,y,z)$ be a nontrivial solution in $K$ to the Diophantine equation 
\[
Ax^p+By^p+Cz^p=0
\]
when $ABC=\pm 1$. Then, we can construct the hyperelliptic curve
\[
Y^2 = X^p+\dfrac{1}{4}\Leftrightarrow Y^2 = X^p+2^{2p-2}
\]
On the other hand, let $(X,Y)$ be a point on $Y^2 = X^p+2^{2p-2}$. Then, we can construct a triplet $(x,y,z)$ as
\[
(x,y,z)=((2^{2p-2})^2XY,-2^{2p-2}X^pY,2^{2p-2}XY^2)
\]
which is a nontrivial solution in $K$ to $(\ref{eq:00})$, e.g.,
\begin{align*}
Ax^p+By^p+Cz^p&=A((2^{2p-2})^2XY)^p+B(-2^{2p-2}XY^2)^p+C(2^{2p-2}X^pY)^p\\
&= \pm1(2^{2p-2})^2XY)^p(2^{2p-2}-Y^2+X^p)\\
&=0
\end{align*}
In this particular case, we have a complete description of all possible solutions and points for our objects of study in any quadratic field.

\section*{Acknowledgements}
\thispagestyle{empty}
The first author was partially supported by SNI - CONACYT. The second author was supported by the PhD Grant 783130 CONACYT - Government of México.

\end{document}